\documentclass{amsart}
\usepackage[utf8]{inputenc}
\usepackage[english]{babel}
\usepackage{amsmath, amsfonts, amssymb, amsthm, stmaryrd}
\usepackage[colorlinks = true,
linkcolor = black,
urlcolor = black,
bookmarksopen = true]{hyperref}

\usepackage{graphicx}
\usepackage[a4paper,top=3cm,bottom=2cm,left=3.5cm,right=3.5cm,marginparwidth=1.75cm]{geometry}

\newtheorem{theorem}{Theorem}
\newtheorem{prop}{Proposition}
\newtheorem{rem}{Remark}
\newtheorem{lemma}{Lemma}

\newtheorem{exam}{Example}
\newtheorem{cor}{Corollary}
\newtheorem{hyp}{Hypothesis}

\newtheorem{innercustomgeneric}{\customgenericname}
\providecommand{\customgenericname}{}
\newcommand{\newcustomtheorem}[2]{%
	\newenvironment{#1}[1]
	{%
		\renewcommand\customgenericname{#2}%
		\renewcommand\theinnercustomgeneric{##1}%
		\innercustomgeneric
	}
	{\endinnercustomgeneric}
}

\newcustomtheorem{customthm}{Theorem}

\title{Heights and transcendence of 
	$p$--adic continued fractions}
\author{Ignazio Longhi, Nadir Murru, Francesco M. Saettone}
\address{Department of Mathematics, University of Torino, Italy}
\email{ignazio.longhi@unito.it}
\address{Department of Mathematics, University of Trento, Italy}
\email{nadir.murru@unitn.it}
\address{Department of Mathematics, Ben-Gurion University of the Negev, Israel}
\email{saettone@post.bgu.ac.il}

\date{}
\begin{document}
	
	\maketitle
	
	\begin{abstract}
		Special kinds of continued fractions have been proved to converge to transcendental real numbers by means of the celebrated Subspace Theorem. In this paper we study the analogous $p$--adic problem. More specifically, we deal with Browkin $p$--adic continued fractions. First we give some new remarks about the Browkin algorithm in terms of a $p$--adic Euclidean algorithm. Then, we focus on the heights of some $p$--adic numbers having a periodic $p$--adic continued fraction expansion and we obtain some upper bounds. Finally, we exploit these results, together with $p$--adic Roth-like results, in order to prove the transcendence of three
		families of $p$--adic continued fractions.
	\end{abstract}
	
	
	\maketitle
	
	\section{Introduction}
	
	Continued fractions are one of the most successful methods to construct transcendental numbers. The first studies in this direction are due to Liouville \cite{Liouville} who dealt with unbounded partial quotients. 
	Later on, Maillet \cite{Maillet} and Baker \cite{baker} exhibited continued fractions with bounded partial quotients converging to transcendental numbers. Furthermore, Baker's results have been recently improved by Adamczewski and Bugeaud \cite{AB07}. The continued fractions studied in these works are \emph{quasi}-periodic, in the sense that they have the form
	\[a_0 + \cfrac{1}{a_1+ \cfrac{1}{a_2+ \cfrac{1}{\ddots}}} = [a_0, a_1, a_2, \ldots]\]
	where $a_i = a_{i+1}$ holds for infinitely many $i$. 
	In \cite{ABD}, the authors proved that the real number $[0, a_1, a_2, \ldots]$ with $a_i = 1 + (\lfloor i \theta \rfloor \bmod k)$, where $0 < \theta < 1$ is irrational and $k \geq 2$ an integer, is transcendental for any $i\ge 1$. This improves a result of Davison \cite{Davison}. 
	Adamczewski and Bugeaud \cite{AB} showed that if a continued fraction begins with an arbitrarily long palindromic string, then it converges to a transcendental number. Further results on this topic can be found in \cite{sturm, Bug, Davison2, Free, Hone}. The main tool used in the above works is the celebrated Roth's theorem and subsequent developments due to Schmidt and later Evertse. 
	This strategy goes as follows: describe a real number $\alpha$ by its continued fraction expansion and use the latter to obtain infinitely many sufficiently good rational approximations so that, by Roth's theorem, $\alpha$ cannot be algebraic.
	
	Over the years, many people have worked to translate these studies into the field of $p$--adic numbers. The first idea for a continued fraction algorithm over $\mathbb Q_p$ is essentially due to Mahler \cite{Mahler}, whose definition was later improved in particular by Ruban \cite{Ruban} and Browkin \cite{Bro78, Bro00}, leading to different algorithms. A significant difference between them is that Ruban's approach leads to a periodic expansion for some rational numbers, while Browkin's continued fractions are always finite for $\alpha\in\mathbb Q$, precisely like in the classical archimedean version. We also remark that, despite the existence of several definitions of continued fractions in $\mathbb Q_p$, a completely satisfactory $p$--adic counterpart is still missing. The main problem is that it is not known if there exists a $p$--adic continued fraction algorithm which eventually becomes periodic if the input is a quadratic irrational. In other words, there is no analogue of Lagrange's theorem. Actually, Lagrange's theorem fails with respect to Ruban's algorithm \cite{ooto} (see also \cite{CVZ} for an effective criterion of periodicity), while for Browkin's algorithm it remains an open question and only some partial results have been found: see for instance \cite{Bed} and \cite{CMT}.
	In Section \ref{s2.1} we will discuss some more reasons to think of Browkin's approach as the ``right'' one, in the spirit of Ostrowski's theorem.
	
	The main goal of this paper is the construction of transcendental $p$--adic continued fractions by means of Browkin's algorithm. 
	In recent works there has been an increasing interest towards many  families of Ruban's $p$--adic continued fractions converging to transcendental numbers. Ooto \cite{ooto} obtained results inspired by Baker's, and the authors  of \cite{Bel} proved that some quasi--periodic and palindromic $p$--adic continued fractions converge to transcendental numbers. 
	
	Our main results are summarized in the following theorem.
	
	\begin{customthm}{A} \label{thm:A}
		Let $$\alpha = [0, b_1, b_2, b_3, \ldots]$$ be a non--periodic Browkin $p$--adic continued fraction. Assume that one of the following holds: \begin{itemize}
			
			\item[(a)] the sequence $(\lvert b_i \rvert_p)_{i \geq 1}$ is bounded and there are infinitely many subsequences 
			$$b_{n_i} = \ldots = b_{n_i + \lambda_i k_i -1} = p^{-1}$$ 
			(with further technical conditions on $n_i$'s, $\lambda_i$'s and $k_i$'s, to be specified in Theorem \ref{qper} and Hypothesis \ref{hyp:1});
			\item[(b)] the sequence $(b_i)_{i\ge 1}$ begins with arbitrarily long palindromes, $|b_n|_\infty > C$ (for some constant $C > 0$ and $|b_n|_p \geq p^4$ for all $n \gg 0$ and $\alpha$ is not a quadratic irrational.
		\end{itemize}
		Then $\alpha$ is transcendental. Moreover, if the hypothesis of {\em (a)} is weakened to $b_{h+k_i} = b_{h}$, for $n_i \leq h \leq n_i + (\lambda_i-1)k_i -1$, for every $i$ (with small changes in the technical conditions, see Theorem \ref{ooto}), then $\alpha$ is either transcendental or quadratic irrational.
	\end{customthm}
	
	Part (a) corresponds to Baker \cite{baker} and Ooto \cite{ooto}, part (b) to Adamczewski-Bugeaud \cite{AB}.
	The proofs are given in Section \ref{sec:main} (Theorems \ref{qper}, \ref{ooto} and \ref{pal}). 
	We also note that the above constructions yield uncountable many transcendental numbers.
	
	In Section \ref{sec:aux} we study the height of $\mathit{periodic}$ continued fractions and provide some upper bounds for it. These results,  stronger than their analogues in \cite{ooto} and of general interest in the context of bounding the height of algebraic numbers, play a crucial role in the proofs of our main theorems. The other tools we are going to use are the $p$--adic versions of Roth's theorem and the Subspace Theorem.

	\subsection*{Acknowledgements}
	We wish to thank Yann Bugeaud for helpful suggestions, Zev Rosengarten for his careful reading and Lea Terracini for useful comments. Thanks to the referees for valuable suggestions about various versions to the paper.
	
	The second author acknowledge support from Ripple’s University Blockchain Research Initiative. The third author is supported by ISF grant 1963/20 and BSF grant 2018250.

	\section{Preliminaries}
	\label{sec:pre}
	
	Let $p$ be an odd prime. We denote by $v_p(\cdot)$, $\lvert \cdot \rvert_p$ and $\lvert \cdot \rvert_\infty$ the $p$--adic valuation, the $p$--adic absolute value and the Euclidean absolute value, respectively. 
	
	Given $\alpha_0 \in \mathbb Q_p$, a $p$--adic continued fraction expansion $[b_0, b_1, \ldots]$ is obtained by choosing a function $s: \mathbb Q_p \rightarrow \mathbb Q$ and then iterating the following steps
	\begin{equation} \label{eq:dfn0} \begin{cases} b_i = s(\alpha_i) \cr \alpha_{i+1} = \cfrac{1}{\alpha_i - b_i} \end{cases} \end{equation}
	for all $i \geq 0$ and $\alpha_i \not= b_i$; if $\alpha_i=b_i$, the process stops and the continued fraction expansion is finite. The function $s$ is defined so to have $|x-s(x)|_p<1$ for every $x\in\mathbb Q_p$ and plays the same role of the floor function in the classical case. In Browkin's algorithm  the idea is to take the integers between $-(p-1)/2$ and $(p-1)/2$ as a set of representatives for $\mathbb Z$ modulo $p$. Then every $\alpha\in\mathbb Q_p$ is represented uniquely as a power series with coefficients in this set and $s$ is defined as
	\begin{align} \label{eq:bro} 
		s\colon &\mathbb Q_p \longrightarrow \mathbb Q \nonumber \\
		&\alpha= \sum_{i=r}^\infty a_ip^i \mapsto \sum_{i=r}^0 a_ip^i \;\;\text{ with }r \in \mathbb Z\text{ and }a_i\in\left\{-\frac{p-1}{2}, \ldots, \frac{p-1}{2}\right\}.
	\end{align}
	In particular, we have $s(\alpha)=0$ if $\alpha$ is in the maximal ideal of the $p$-adic integers and $|s(\alpha)|_p=p^k$ for some $k\in\mathbb N$ otherwise. 
	
	The $b_i$'s are called \emph{partial quotients} and the $\alpha_i$'s \emph{complete quotients}.
	We define the sequences $(A_i)_{i \geq -2}$ and $(B_i)_{i \geq -2}$ as follows:
	\begin{equation} \label{eq:dfnAB} \begin{cases} A_{-2} = 0, A_{-1} = 1 \cr A_i = b_iA_{i-1} + A_{i-2}, \forall i \geq 0  \end{cases}, \quad \begin{cases} B_{-2} = 1, B_{-1} = 0 \cr B_i = b_iB_{i-1} + B_{i-2}, \forall i \geq 0  \end{cases} \end{equation}
	so that their ratios give the convergents of the $p$--adic continued fraction, i.e., $\frac{A_i}{B_i} = [b_0, \ldots, b_i]$, for all $i \geq 0$. The classical identities
	\begin{equation} \label{eq:id-base} 
		\alpha_0 = [b_0, \ldots, b_i, \alpha_{i+1}] = \cfrac{\alpha_{i+1}A_i+A_{i-1}}{\alpha_{i+1}B_i+B_{i-1}}, \quad A_{i-1}B_i - B_iA_{i-1} = (-1)^i
	\end{equation}
	still hold for $p$--adic continued fractions, since they are obtained in a formal way.
	
	In the following proposition, we summarize some well-known facts about Browkin $p$--adic continued fractions.
	
	\begin{prop}
		\label{prop:varie}
		Given $\alpha_0 = [b_0, b_1, \ldots]$, we have
		\begin{enumerate}
			\item[(i)] $v_p(b_{n}) = v_p(\alpha_n)$, i.e., $\lvert b_n \rvert_p = \lvert \alpha_n \rvert_p$, for all $n \geq 1$.
			\item[(ii)] $v_p(b_n) < 0$, i.e., $\lvert b_n \rvert_p \geq p$, for all $n \geq 1$.
			\item[(iii)] $v_p(A_n) = v_p(b_0) + \ldots + v_p(b_n)$, $\lvert A_n \rvert_p = \lvert b_0 \cdots b_n \rvert_p$, for all $n \geq 0$, if $b_0 \not=0$.
			\item[(iv)] $v_p(A_n) = v_p(b_2) + \ldots + v_p(b_n)$, $\lvert A_n \rvert_p = \lvert b_2 \cdots b_n \rvert_p$, for all $n \geq 2$, if $b_0=0$.
			\item[(v)] $v_p(B_n) = v_p(b_1) + \ldots + v_p(b_n)$, $\lvert B_n \rvert_p = \lvert b_1 \cdots b_n \rvert_p$, for all $n \geq 1$.
			\item[(vi)] $v_p\left( \alpha_0 - \cfrac{A_n}{B_n} \right) = -v_p(B_n B_{n+1})$, for all $n \geq 0$.
		\end{enumerate} 
	\end{prop}
	\begin{proof}
		The proofs of claims (i), (ii), (v) and (vi) can be found in \cite[section 2]{Bro78}. As for (iii) and (iv), they are easily checked for $n=0$, $n=2$ respectively and if they hold up to $n$ then
		$$|A_{n+1}|_p=|b_{n+1}A_n+A_{n-1}|_p=|b_{n+1}A_n|_p$$ 
		since $|b_{n+1}A_n|_p>|A_n|_p\geq|A_{n-1}|_p$ by point (ii), so one can conclude by induction.
	\end{proof}
	
	\begin{rem}\emph{
			The $A_i$'s and $B_i$'s are rational numbers whose denominators are powers of $p$ depending on the valuations of the partial quotients.}
	\end{rem}
	
	\subsection{Some remarks on Browkin's approach to $p$-adic continued fractions}\label{s2.1}
	For any prime $\ell$, let $|\cdot|_\ell$ and $\bar B_\ell(0,1)=\{x\in\mathbb Q:|x|_\ell\leq1\}$ denote respectively the $\ell$-adic absolute value and the closed unit ball in $\mathbb Q$ with respect to $|\cdot|_\ell$. Then we have
	\[\mathbb Z=\bigcap_{\ell\text{ prime}}\bar B_\ell(0,1)\;\;\text { and }\;\;\mathbb Z[p^{-1}]=\bigcap_{\ell\neq p}\bar B_\ell(0,1)\]
	(the sets of $S$-integers of $\mathbb Q$, with $S$ being respectively $\{\infty\}$ and $\{\infty,p\}$).
	Moreover, putting $B_\infty(0,r)=\{x\in\mathbb Q:|x|_\infty<r\}$ for any $r>0$, a simple computation 
	(as performed in \cite[page 69]{Bro78}) shows
	\begin{equation} \label{eq:bro2}
		s(\mathbb{Q}_p)\subseteq\left\{x\in\mathbb Z[p^{-1}]:|x|_\infty<\frac{p}{2}\right\}=B_\infty\!\left(0,\frac{p}{2}\right)\cap\bigcap_{\ell\neq p}\bar B_\ell(0,1)\,.    
	\end{equation}
	The value $p/2$ for the radius of the archimedean ball is optimal (see again \cite[page 69]{Bro78}).
	
	Thus, both the classical floor function and Browkin's definition of $s$ can be seen as instances of a function $s_v\colon\mathbb Q_v\rightarrow\mathbb Q$ (with $v$ a place of $\mathbb Q$ and $\mathbb Q_v$ the completion with respect to $v$) such that \begin{itemize}
		\item[(a)] $|x-s_v(x)|_v<1$ for every $x\in\mathbb Q_v$;
		\item[(b)] the function $s_v$ takes values in the closed unit balls at all places outside $S=\{\infty,v\}$ (i.e., $s_v(\mathbb Q_v)$ is a subset of the ring of $S$-integers);
		\item[(c)] if $v\neq\infty$, then $s_v(\mathbb Q_v)$ is contained in a ball $B_\infty(0,r)$ with small radius (as small as possible, compatibly with the previous conditions).
	\end{itemize}
	Condition (c) is what differentiates Browkin's definition from alternative approaches, like Ruban's (which admits greater values for $|s(x)|_\infty$), and in our opinion is the key reason for a closer analogy with the real case. In Ruban's algorithm, the map $s$ is still defined as in  \eqref{eq:bro}, but with the coefficients $a_i$ chosen in the set $\{0, \ldots, p-1\}$.  One consequence is that Ruban's continued fraction expansion of a rational number $\alpha$ need not be finite, even when $\alpha$ is an integer (e.g., take $\alpha=-p$; see \cite[page 303]{laohak} for details). With Browkin's choice, instead, a $p$-adic number is rational if and only if it has a finite continued fraction expansion, like it happens with the reals. Equivalently, Browkin's continued fractions correspond to a $p$-adic version of the Euclidean algorithm with the property of terminating on $\mathbb Q$ (analogous to the classical property  on $\mathbb Z$), as we are going to illustrate. 
	
	\begin{prop}
		\label{prop:euclide}
		Given any $x$ and $y$ in $\mathbb Q_p$, with $y \not=0$, there exist unique $q \in \mathbb Z\left[\frac{1}{p}\right]$ and $r \in \mathbb Q_p$ with $\lvert q \rvert_\infty <\frac{p}{2}$ and $\lvert r \rvert_p < \lvert y \rvert_p$ such that $x = q y + r$.
	\end{prop}
	\begin{proof}
		See \cite[Theorem 1]{Lager09}.
	\end{proof}
	
	Recall that, for $\alpha \in \mathbb Q_p$, Browkin's $s$ yields $s(\alpha) \in \mathbb Z[p^{-1}]$, $|\alpha-s(\alpha)|_p<1$ and $ |s(\alpha)|_\infty < \frac{p}{2}$. Thus comparing the two sides of
	\[  s\left( \frac{x}{y} \right) +\left(\frac{x}{y}-s\left( \frac{x}{y} \right)\right)= \frac{x}{y} = q + \frac{r}{y}  \]
	we see that the quotient $q$ defined in Proposition \ref{prop:euclide} coincides with $s(x/y)$.
	
	
	As described in \cite[Definition 2]{Lager09}, the $p$-adic Euclidean algorithm is obtained by iteration of the division-with-remainder process described in Proposition \ref{prop:euclide}. If we start with inputs $x,y$ both in $\mathbb Q$, then the algorithm stops (i.e., the remainder becomes $0$) after a finite number of steps \cite[Theorem 2]{Lager09}. In Browkin's algorithm, the continued fraction of $x/y$ results by iterating the same steps.  Indeed, putting $x=x_0, y=x_1$, with  $\lvert x_0 \rvert_p \geq \lvert x_1 \rvert_p \not= 0$, we can write $x_0 = b_0 x_1 + x_2$ where $b_0 = s(x_0 / x_1)$ and $\lvert x_2 \rvert_p < \lvert x_1 \rvert_p$. Repeating this, we get
	\[ x_i = b_i x_{i+1} + x_{i+2}, \]
	for $i = 0, 1, \ldots$, where $b_i = s(x_i / x_{i+1})$. Thus, setting $\alpha_i = \frac{x_i}{x_{i+1}}$, we obtain the continued fraction expansion of $\alpha_0$, since we have
	\[\alpha_i = b_i + \cfrac{1}{\alpha_{i+1}}\]
	and the $b_i$'s and $\alpha_i$'s are obtained as in \eqref{eq:bro}.  
	
	Finally, we note that Browkin's proof that the continued fraction of $\alpha\in\mathbb Q$ is finite (\cite[Theorem 3]{Bro78}) and Lager's proof that the $p$-adic Euclidean algorithm stops if one starts with $x,y\in\mathbb Q$ (\cite[Theorem 2]{Lager09}) are both based on the same computation.
	
	\begin{rem}
		\emph{Let $K$ be a number field, $S_\infty$ the set of its non-archimedean places and $\mathfrak P$ a non-archimedean place of $K$. The question of finiteness in a $\mathfrak P$-adic continued fraction expansion for elements of $K$ is discussed at length in \cite{CMT1}. Here we want to mention that a sufficient condition for ensuring such finiteness arises exactly from adapting condition (c) above: that is, choosing a ``floor function'' $s_{\mathfrak P}$ so that its values are $S$-integers (with $S=S_\infty\cup\{\mathfrak P\}$) which are ``small'' at the archimedean places. See \cite[Theorems 4.5 and 4.6]{CMT1} for precise statements. (One should also compare our conditions (a) and (b) with the more precise axiomatization of $\mathfrak P$-adic floor function in \cite[Definition 3.1]{CMT1}.)  }
	\end{rem}
	
	\subsubsection{On the axiomatization of $s$}
	Let $K$ be a field complete with respect to a discrete valuation $\nu$ and let $\frak m$ be the maximal ideal in its ring of integers. In \cite[section 1]{Bro78} one can find an axiomatic definition of a function $s\colon K\rightarrow K$ which should play the role of the floor function for continued fractions in $K$: namely, $s$ is required to satisfy the three properties
	\begin{itemize}
		\item[(S1)] $s(0)=0$\,;
		\item[(S2)] $\nu\circ s=\nu$\,;
		\item[(S3)] $(a-b)\in\frak m\Longrightarrow s(a)=s(b)$.
	\end{itemize}
	However, there are some problems with this definition. To begin with,  (S2) and (S3) are not compatible (since $\nu$ takes infinitely many distinct values in $\frak m$). Moreover, these conditions do not imply $x-s(x)\in\frak m$ (i.e., the analogue of our condition (a) above): indeed, the function $s\colon\mathbb Q_p\rightarrow\mathbb Q_p$ defined by
	$$s(a)=\begin{cases}0\;\;\;\;\;\;\;\;\text{ if }a\in p\mathbb Z_p\\p^{v_p(a)}\;\text{ otherwise}\end{cases}$$
	satisfies (S1), (S3) and a modified version of (S2), but $2p^{-1}-s(2p^{-1})$ is not in $p\mathbb Z_p$. It is clear from the rest of the paper that what the author had really in mind when stating (S2) and (S3) was the condition $s=s_1\circ\pi$, where $\pi\colon K\rightarrow K/\frak m$ is the quotient map and $s_1$ a section of $\pi$.
	
	\section{Auxiliary results}
	\label{sec:aux}
	
	We collect some auxiliary results needed for proving the main theorems in the next section. In particular, we give some bounds for the (naive) height of algebraic numbers with a specific expansion in Browkin $p$--adic continued fractions. These results may be also of general interest independently of their use. Moreover, we state the $p$--adic version of the Subspace Theorem, which is an essential tool for proving the transcendence of numbers also in the real case. We also give two weaker versions of it (Corollaries \ref{coroth} and  \ref{coschmidt}), which are eventually used in our proofs.
	
	\subsection{Properties of partial quotients, convergents and height bounds}
	
	In the following, unless otherwise stated, all continued fractions developments are based on Browkin's choice of $s$, as in \eqref{eq:bro}, and the sequences $(A_i),(B_i)$ are always defined by \eqref{eq:dfnAB}.
	
	\begin{lemma}\label{1}
		If $\alpha, \beta \in \mathbb Q_p$ have the same $n + 1$ first partial quotients in the $p$--adic continued fraction expansion, then $\lvert \alpha - \beta \rvert_p < \cfrac{1}{\lvert B_n \rvert_p^2}$.
	\end{lemma}
	\begin{proof}
		By hypothesis, the $n$--th and $(n+1)$--th convergents of the $p$-adic continued fraction of $\alpha$ and $\beta$ are the same, namely $\frac{A_n}{B_n}$ and $\frac{A_{n+1}}{B_{n+1}}$. By Proposition \ref{prop:varie}, we have
		\[\left\lvert \alpha - \cfrac{A_n}{B_n}  \right\rvert_p = \cfrac{1}{\lvert B_{n+1} \rvert_p \lvert B_n \rvert_p} = \cfrac{1}{\lvert b_{n+1} \rvert_p \lvert B_n \rvert_p^2}\]
		since
		\(v_p(B_{n+1}) = v_p(b_{n+1}B_n + B_{n-1}) = v_p(b_{n+1}) + v_p(B_n).\)
		The same holds for $\beta$ and thus we conclude.
	\end{proof}
	
	\begin{lemma} \label{lemma:inf-p}
		Given $\alpha_0 = [b_0, b_1, \ldots]$, we have $\lvert A_n \rvert_\infty < \lvert A_n \rvert_p$ and $\lvert B_n \rvert_\infty < \lvert B_n \rvert_p$, for all $n\gg0$.
	\end{lemma}
	
	\begin{proof}
		Let $x_0,x_1\in\mathbb Q$ be arbitrary and put 
		$$x_n:=b_nx_{n-1}+x_{n-2}\,.$$
		Choose $M$ so that $|x_0|_\infty<M$ and $\displaystyle|x_1|_\infty<M\left(\frac{p}{2}+1\right)$ both hold. Then 
		\begin{equation} \label{eq:lmm2bn} |x_k|_\infty<M\left(\frac{p}{2}+1\right)^{\!k} \end{equation}
		is true for every $k\in\mathbb N$. Indeed, assuming by induction that \eqref{eq:lmm2bn} is satisfied up to $n-1$, one finds
		\begin{eqnarray*} |x_n|_\infty & \le & |b_n|_\infty\,|x_{n-1}|_\infty+|x_{n-2}|_\infty  \\
			& \leq & \frac{p}{2}|x_{n-1}|_\infty+|x_{n-2}|_\infty\hspace{70pt}\text{ by \eqref{eq:bro2} } \\
			& < & \frac{p}{2} M\left(\frac{p}{2}+1\right)^{\!n-1}+M\left(\frac{p}{2}+1\right)^{\!n-2}=M\left(\frac{p}{2}+1\right)^{\!n-2}\left(\frac{p^2}{4}+\frac{p}{2}+1\right) \\
			& < & M\left(\frac{p}{2}+1\right)^{\!n}\,. \end{eqnarray*}
		In particular, one can choose $M$ such that that \eqref{eq:lmm2bn} applies to both sequences $(A_n)$ and $(B_n)$. Now it is enough to recall Proposition \ref{prop:varie} and note that 
		$$M\left(\frac{p}{2}+1\right)^{\!n}<p^{n-2}\leq\min\big\{|A_n|_p,|B_n|_p\big\}$$ 
		is true for $n\gg0$.
	\end{proof}
	
	\begin{rem}\label{rem:AB}
		\emph{ With slightly more effort, one can check that $\lvert B_n \rvert_\infty \leq \lvert B_n \rvert_p$ holds for all $n\geq-2$. If $|b_0|_p\neq1$, one also has $\lvert A_n \rvert_\infty \leq \lvert A_n \rvert_p$ for all $n \geq -2$.}
	\end{rem}
	
	We recall the following classical definition.
	Let $\alpha$ be an algebraic number of degree $d$: then $\alpha$ is a root of a non-zero polynomial with integer coefficients  $c_0,...,c_d$. The {\em naive height} of $\alpha$ is defined as 
	$$h(\alpha) = \frac{\max \{ \lvert c_i\rvert_\infty: i = 0, \ldots, d \}}{\gcd(c_0, \ldots, c_d)}\,.$$
	
	\begin{lemma} \label{lemma:h1}
		Assume $\alpha\in\mathbb Q_p$ has a periodic Browkin continued fraction
		$$\alpha = [0, b_1, \ldots, b_{k},\overline{b_{k+1}, \ldots, b_{k+t+1}}].$$ 
		Then
		$$h(\alpha) \leq \frac{8}{p^2}\lvert B_{k+t+1} \rvert_p^2 \lvert B_k \rvert_p^2\,.$$
	\end{lemma}
	
	There is no loss of generality in assuming $k \geq 1$, since one can always start the period after the first few terms. 
	
	\begin{proof}
		Let us consider $\beta = [\overline{b_{k+1}, \ldots, b_{k+t+1}}]$ and attach to it sequences $(\tilde A_i),(\tilde B_i)$ by \eqref{eq:dfnAB}. 
		Using \eqref{eq:id-base}, we get
		\[ \beta = \cfrac{\beta \tilde A_t + \tilde A_{t-1}}{\beta \tilde B_t + \tilde B_{t-1}}, \quad \alpha = \cfrac{\beta A_k + A_{k-1}}{\beta B_k + B_{k-1}} \]
		from which
		\begin{equation} \label{eq:poly}
			C_0 \alpha^2 + C_1 \alpha + C_2 = 0 
		\end{equation}
		where
		\tiny{
			\[ \begin{cases}
				C_0 = - \tilde A_{t-1} B_k^2 + \tilde A_t B_k B_{k-1} - \tilde B_{t-1} B_k B_{k-1} + \tilde B_t B_{k-1}^2 \cr
				C_1 = 2 \tilde A_{t-1} A_k B_k - \tilde A_{t} A_{k-1} B_{k} + \tilde B_{t-1} A_{k-1} B_k - \tilde A_t A_k B_{k-1} + \tilde B_{t-1}A_k B_{k-1} - 2 \tilde B_t A_{k-1} B_{k-1} \cr
				C_2 = -\tilde A_{t-1} A_k^2 + \tilde A_t A_k A_{k-1} -\tilde B_{t-1} A_k A_{k-1} + \tilde B_t A_{k-1}^2
			\end{cases}\]
		}
		
		\normalsize{}
		By construction we have 
		\[ C_i = \sum_{j=1}^{h_i} C_{ij} \in \mathbb Z\left[p^{-1}\right] \]
		where $h_0=h_2=4$ or $h_1=8$ and the summands $C_{ij}$ are given by the formulae above. In order to bound $h(\alpha)$ by \eqref{eq:poly}, we need to multiply by a power of $p$, so to have integer coefficients. 
		
		Putting
		\begin{equation*} \label{eq:norme} 
			\rvert \tilde A_i \lvert_p = p^{\tilde e_i}, \quad \rvert \tilde B_i \lvert_p = p^{\tilde f_i}, \quad \rvert A_i \lvert_p = p^{e_i}, \quad \rvert B_i \lvert_p = p^{f_i}, 
		\end{equation*}
		we obtain four increasing  sequences of integers $(\tilde e_i), (\tilde f_i), (e_i)$ and $(f_i)$. By Proposition \ref{prop:varie}, we have $\tilde e_i > \tilde f_i>0$ for $i \geq 0$, because $|b_{k+1}|_p\geq p$, and $f_i > e_i>0$ for $i \geq 2$, since $b_0 = 0$; moreover $\tilde e_i + f_k=f_{k+i+1}$ holds for $i\geq0$. By these observations, one can check that 
		\[|C_{ij}|_p \leq p^{\tilde e_t + 2f_k -1} = \frac{1}{p} \lvert B_{k+t+1} \rvert_p \lvert B_{k} \rvert_p\]
		for all $i, j$.
		Therefore multiplying \eqref{eq:poly} by $p^{\tilde e_t + 2f_k-1}$ we obtain a polynomial with integer coefficients satisfied by $\alpha$. This implies
		\[ h(\alpha) \leq \frac{1}{p}  \lvert B_{k+t+1} \rvert_p \lvert B_{k} \rvert_p \max \{ \lvert C_0 \rvert_\infty, \lvert C_1 \rvert_\infty, \lvert C_2 \rvert_\infty \}\,. \]
		
		To conclude note that by Lemma \ref{lemma:inf-p} (and Remark \ref{rem:AB}, using $b_0 = 0$ and $|b_{k+1}| \geq p$) we have
		\[\lvert C_i \rvert_\infty \leq \sum_{j=1}^{h_i} |C_{ij}|_p \leq \frac{8}{p} \lvert B_{k+t+1} \rvert_p \lvert B_{k} \rvert_p\,.\]
	\end{proof}
	
	\begin{lemma} \label{lemma:h2}
		For every $k \geq 1$, given $b_i = \frac{\hat b_i}{p^{a_i}}$ such that $a_i, \hat b_i \in \mathbb Z$, $a_i > 0$, $|\hat b_i|_\infty < \sqrt{\frac{3}{14}} \cdot \frac{p^a}{F_{k+1}}$ (for the $b_i \not= p^{-1}$), where $a = \min\{ a_i : 1 \leq i \leq k, a_i \not= 1 \}$  and $(F_i)_{i \geq 0} = (0, 1, 1, 2, 3, 5, \ldots)$ is the Fibonacci sequence, then 
		$\alpha = [0, b_1, b_2, \ldots, b_{k}, \overline{p^{-1}}]$ satisifies
		$$h(\alpha) \leq \lvert B_{k} \rvert_p^2.$$ 
	\end{lemma}
	\begin{proof}
		We can write
		\[\quad A_i = \frac{\hat A_i}{p^{e_i}}, \quad B_i = \frac{\hat B_i}{p^{f_i}} \]
		where $\hat A_i, \hat B_i \in \mathbb Z$, $v_p(\hat b_i) = v_p(\hat A_i) = v_p(\hat B_i) = 0$, $e_i = a_2 + \ldots + a_i$, $f_i = a_1 + \ldots + a_i$, for any $1 \leq i \leq k$. Note that we have
		\[[0, b_1, \ldots, b_i] = \cfrac{A_i}{B_i} = \cfrac{p^{a_1} \hat A_i}{\hat B_i}\]
		for any $1 \leq i \leq k$.
		We have that $\hat B_{i}$ is the sum of $F_{i+1}$ elements of the kind $\hat b_{i_1}\cdots \hat b_{i_h}p^{a_{j_1}}\cdots p^{a_{j_l}}$ such that $\{i_1, \ldots, i_h, j_1, \ldots j_l\} = \{1, \ldots, i\}$ and $h+l = i$, for any $1 \leq i \leq k$. We prove this by induction. 
		For $i=1$, we have 
		$$\cfrac{A_1}{B_1} = \cfrac{p^{a_1} \hat A_1}{\hat B_1}  = [0, b_1] = \cfrac{1}{b_1} = \cfrac{p^{a_1}}{\hat b_1}, \quad \hat B_1 = \hat b_1$$
		For $i = 2$, we have 
		\[ \cfrac{A_2}{B_2} = \cfrac{p^{a_1}\hat A_2}{\hat B_2} = [0, b_1, b_2] = \cfrac{p^{a_1} \hat b_2}{\hat b_1 \hat b_2 + p^{a_1}p^{a_2}}, \quad  \hat B_2 = \hat b_1 \hat b_2 + p^{a_1}p^{a_2}. \]
		For any $3 \leq i \leq k$, we have
		\[ \cfrac{A_i}{B_i} = \cfrac{b_{i} A_{i-1} + A_{i-2}}{b_{i} B_{i-1} + B_{i-2}} = \cfrac{\cfrac{\hat b_{i}}{p^{a_{i}}} \cdot \cfrac{\hat A_{i-1}}{p^{e_{i-1}}} + \cfrac{\hat A_{i-2}}{p^{e_{i-2}}}}{\cfrac{\hat b_{i}}{p^{a_{i}}} \cdot \cfrac{\hat B_{i-1}}{p^{f_{i-1}}} + \cfrac{\hat B_{i-2}}{p^{f_{i-2}}}} = \cfrac{p^{a_1}(\hat b_{i} \hat A_{i-1} + p^{a_{i-1}}p^{a_i} \hat A_{i-2})}{\hat b_i \hat B_{i-1} +  p^{a_{i-1}}p^{a_i}\hat B_{i-2}}. \]
		By inductive hypothesis, $\hat B_{i-1}$ has $F_{i}$ addends of the kind  $\hat b_{i_1}\cdots \hat b_{i_h}p^{a_{j_1}}\cdots p^{a_{j_l}}$ with\\  $\{i_1, \ldots, i_h, j_1, \ldots j_l\}$ $= \{1, \ldots, i-1\}$ and $h+l = i-1$. Thus, $\hat b_{i} \hat B_{i-1}$  has $F_{i+1}$ addends of the kind  $\hat b_{i_1}\cdots \hat b_{i_h}p^{a_{j_1}}\cdots p^{a_{j_l}}$ with  $\{i_1, \ldots, i_h, j_1, \ldots j_l\}$ $= \{1, \ldots, i\}$ and $h+l = i$. 
		
		Similarly, $p^{a_{i-1}}p^{a_i} \hat B_{i-2}$  has $F_{i-1}$ addends of the kind  $\hat b_{i_1}\cdots \hat b_{i_h}p^{a_{j_1}}\cdots p^{a_{j_l}}$ with\\  $\{i_1, \ldots, i_h, j_1, \ldots j_l\} = \{1, \ldots, i\}$ and $h+l = i$. 
		
		Finally, we have that $\hat B_i$ has $F_{i}+ F_{i-1} = F_{i+1}$ addends of the kind  $\hat b_{i_1}\cdots \hat b_{i_h}p^{a_{j_1}}\cdots p^{a_{j_l}}$ with  $\{i_1, \ldots, i_h, j_1, \ldots j_l\} = \{1, \ldots, i\}$ and $h+l = i$. 
		
		By hypothesis we have $\lvert \hat b_i \rvert_\infty < \sqrt{\cfrac{3}{14}}\cdot \cfrac{p^a}{F_{k+1}}$, where $a = \min\{ a_i : 1 \leq i \leq k, a_i\not= 1 \}$ or $\hat b_i = 1 < p$ (if $b_i = p^{-1}$) for any $1 \leq i \leq k$, then $\lvert B_{k} \rvert_\infty < \sqrt{\frac{p}{4p + 2}}$. 
		Indeed, $B_{k} = \cfrac{\hat B_{k}}{p^{f_{k}}}$ and we have seen that $\hat B_{k}$ is a sum of $F_{k+1}$ elements of the kind $\hat b_{i_1}\cdots \hat b_{i_h}p^{a_{j_1}}\cdots p^{a_{j_l}}$ with  $\{i_1, \ldots, i_h, j_1, \ldots j_l\} = \{1, \ldots, i\}$ and $h+l = k$, that is
		\[ B_k = \sum \cfrac{\hat b_{i_1} \cdots \hat b_{i_h}}{p^{a_1} \cdots p^{a_{h}}} \]
		where there are $F_{k+1}$ addends. Thus, by hypothesis we have that any addend satisfies
		\[ \hat b_{i_1}\cdots \hat b_{i_h} < \sqrt{\cfrac{3}{14}} \cdot \cfrac{p^{a_{i_1}+\ldots+a_{i_h}}}{F_{k+1}}, \]
		from which we obtain
		\[ \lvert B_k \rvert_\infty < \sqrt{\cfrac{3}{14}} \leq \sqrt{\cfrac{p}{4p+2}} \]
		for every odd prime $p$.
		Similar arguments can be also used  to show that $\lvert A_{k} \rvert_\infty < \sqrt{\frac{p}{4p + 2}}$.
		Let us observe that, fixed an integer $k$, we can take the partial quotients $b_1, \ldots, b_k$ such that $\lvert \hat b_i \rvert_\infty < \sqrt{\cfrac{3}{14}}\cdot \cfrac{p^a}{F_{k+1}}$ in infinitely many different ways. Indeed, it is sufficient to take the minimum of the $a_i$'s sufficiently large and observe that the integers $a_i$'s can be chosen in infinitely different ways.
		Now, consider $\beta = [\overline{p^{-1}}]$, we have that $\beta$ is a root of $p \beta^2 - \beta - p = 0$. Moreover, we have
		\[ \alpha = \cfrac{\beta A_k + A_{k-1}}{\beta B_k + B_{k-1}}, \quad \beta = \cfrac{A_{k-1} - \alpha B_{k-1}}{B_k \alpha - A_k}. \]
		Thus, we obtain the minimal polynomial of $\alpha$ as
		\begin{align*} 
			&(B_kB_{k-1}-B_k^2p+B_{k-1}^2p)\alpha^2 + 
			\\
			(2A_kB_k-2A_{k-1}B_{k-1}p-&A_kB_{k-1}-A_{k-1}B_k)\alpha +(A_kA_{k-1}-A_k^2p+A_{k-1}^2p) = 0, 
		\end{align*}
		which can be also written as
		\begin{align*} 
			&(\hat B_k \hat B_{k-1}p^{a_k-1}-\hat B_k^2+\hat B_{k-1}^2p^{2a_k})\alpha^2 +
			\\
			(2\hat A_k \hat B_k &p^{a_1}- 2 \hat A_{k-1} \hat B_{k-1} p^{a_1+2a_k} - \hat A_k \hat B_{k-1}p^{a_1+a_k-1} -  \hat A_{k-1}\hat B_k p^{a_1+a_k-1})\alpha +
			\\
			&(\hat A_k \hat A_{k-1} p^{2a_1+a_k-1}-\hat A_k^2 p^{2a_1}+ \hat A_{k-1}^2 p^{2a_1+2a_k}) = 0.
		\end{align*}
		Under the condition $\lvert A_{k} \rvert_{\infty}, \lvert B_{k} \rvert_{\infty} < \sqrt{\frac{p}{4p + 2}}$, we get
		\begin{itemize}
			\item $\lvert \hat B_k \hat B_{k-1}p^{a_k-1}-\hat B_k^2+\hat B_{k-1}^2p^{2a_k} \rvert_\infty \leq p^{2(a_1 + \ldots + a_k)}$, since \\ $\lvert B_kB_{k-1} \rvert_{\infty} \cdot \cfrac{1}{p} + B_{k-1}^2 + B_k^2 \leq 1$
			
			\item $\lvert 2\hat A_k \hat B_k p^{a_1}- 2 \hat A_{k-1} \hat B_{k-1} p^{a_1+2a_k} - \hat A_k \hat B_{k-1}p^{a_1+a_k-1} - \hat A_{k-1}\hat B_k p^{a_1+a_k-1} \rvert_{\infty} \leq p^{2(a_1 + \ldots + a_k)}$, since $2\lvert A_kB_k \rvert_{\infty} + 2\lvert A_{k-1}B_{k-1} \rvert_{\infty} + \lvert A_kB_{k-1}\rvert_{\infty}\cdot \cfrac{1}{p}+\lvert A_{k-1}B_k\rvert_{\infty} \cfrac{1}{p} \leq 1$
			
			\item $\lvert\hat A_k \hat A_{k-1} p^{2a_1+a_k-1}-\hat A_k^2 p^{2a_1}+ \hat A_{k-1}^2 p^{2a_1+2a_k}\rvert_{\infty} \leq p^{2(a_1 + \ldots + a_k)}$, since \\ $\lvert A_kA_{k-1} \rvert_\infty \cdot \cfrac{1}{p} + A_k^2 + A_{k-1}^2 \leq 1$. 
		\end{itemize}
		Thus, we have $h(\alpha) \leq p^{2(a_1 + \ldots + a_k)} =  \lvert B_{k} \rvert_p^2$.
	\end{proof}

	\begin{exam}\emph{
			With the notation of Lemma \ref{lemma:h2}, for $p = 5$, $k = 2$ consider the continued fraction
			\[ \alpha = \left[0, \cfrac{4}{5^2}, -\cfrac{3}{5^3}, \overline{\cfrac{1}{5}}\;\right]. \]
			We observe that $\lvert\hat b_1 \rvert_\infty = 4$ and $\lvert\hat b_2 \rvert_\infty = 3$ are smaller than $\sqrt{\cfrac{3}{14}}\cdot \cfrac{5^2}{2} = 5.78\ldots$\;. 
			\\The minimal polynomial of $\alpha$ is $9.129.469 x^2 + 5.530.075 x - 9.713.125$ and $\lvert B_2 \rvert_p^2 = p^{10} = 9.765.625$. Clearly, we can construct infinitely many continued fractions of this kind with $k = 2$. Indeed, we can take $b_1$ and $b_2$ with a denominator consisting of any possible power of $p$ (and large values of this powers allow also more possible choices for $\hat b_1$ and $\hat b_2$).}
	\end{exam}
	
	\subsection{$p$-adic Subspace Theorem}
	We begin restating a $p$-adic version of the celebrated Roth's theorem and its generalizations due to Schmidt, Evertse, Schlickewei; see \cite{roth}. This is currently one of the main technical tools of diophantine approximation applied to archimedean continued fractions, see, for instance, \cite{AB}.

	Let us recall the  definition of the {\it Weil} (sometimes called {\it absolute}) height function $H$ of an algebraic number $x$. The absolute value $|\cdot|_v$ is the unique absolute value on $\mathbb{Q}(x)$ which restricts to the usual $v$-adic absolute value over $\mathbb{Q}$. Let also $n$ and $n_v$ be respectively the global and local degrees of $x$. Then the Weil height is defined as 
	$$H(x):= \prod_v\max\{1,|x|^{n_v}_v\}^{1/n},$$
	for $v$ running over all places of $\mathbb{Q}(x)$.
	
	\begin{rem}\label{rem:H}
		\emph{
			The relation between the naive height $h$ and the absolute height $H$ is well known. For $\alpha\in\mathbb{Q}$, the two heights coincide, on the other hand, for $\alpha\in\overline{\mathbb{Q}}$ of degree $D$, we have
			\[
			H(\alpha)\le (D+1)^{1/2}h(\alpha),\;\; h(\alpha)\le 2^D H(\alpha).
			\]
			For the proof and much more, see \cite[Part B, p.177]{silv}.}
	\end{rem}
	
	We recall the general formulation of Roth's theorem, as stated in \cite{silv}.
	
	\begin{theorem}
		Let $K$ be a number field, let $S$ be a finite set of absolute values on $K$ with each absolute value extended in some way to $\overline{K}$. Let $\alpha \in \overline{ K}$ and $\varepsilon > 0$ be given. Suppose that 
		\[\xi: S \rightarrow [0, 1] \text{ is a function satisfying } \sum_{v \in S} \xi(v) = 1.\]
		Then, there are only finitely many $\beta \in  K$ with the property that $$\lvert \alpha - \beta \rvert_v \leq \cfrac{1}{H(\beta)^{(2+\varepsilon)\xi(v)}}$$
		for all $v \in S$.
	\end{theorem}
	\begin{proof}
		See \cite[Theorem D.2.2]{silv}.
	\end{proof}
	
	In our first main Theorem, we need the following version of Roth's theorem as described in the next corollary.
	
	\begin{cor}\label{coroth}
		Consider $K$ a quadratic field, $S =\{ \infty, p\}$ and $\xi(\infty) = 0$, $\xi(p) = 1$. Let $\varepsilon, C > 0$, then there are finitely many $\beta\in K$ such that 
		\[    |\alpha - \beta|_p\le C H(\beta)^{-2-\epsilon}    \]
		holds. 
	\end{cor}
	
	We now state a $p$--adic version of the Subspace Theorem.
	
	\begin{theorem}\label{roth} 
		
		Let $n \ge2$ and $S=\{\infty, p\}$ contain two places of $\mathbb{Q}$. Consider $L_{1,\infty},...,L_{n,\infty}$ independent linear forms in $n$ variables with real algebraic coefficients and $L_{1,p},...,L_{n,p}$ independent linear forms in $n$ variables with $p$-adic algebraic coefficients.
		Then, for every $\epsilon > 0$ there are a finite number of proper rational subspaces $T_1, \ldots, T_r$ such that every $x \in \mathbb Z[p^{-1}]^n - \{0\}$ satisfying 
		\[
		\prod_{v\in S}
		\prod_{i\le n}|L_{i,v}(x)|_v < \max\{|x_1|_\infty,...,|x_n|_\infty\}^{-\epsilon}\max\{|x_1|_p,...,|x_n|_p\}^{-\epsilon}
		\]
		lies in one of these subspaces.   
	\end{theorem}
	
	\begin{proof}
		See \cite[Corollary 7.2.5]{Bomb}
	\end{proof}
	
	We exploit Theorem \ref{roth} to prove the following results which will be applied for proving part (b) of Theorem \ref{thm:A}.
	
	\begin{lemma} \label{lemma:appl-sub}
		Let $\alpha, \beta \in \mathbb Q_p$ be algebraic and $1, \alpha, \beta$ linearly independent over $\mathbb Q$. Then for every $\epsilon > 0$ there are only finitely many $(x, y, z) \in \mathbb Z[p^{-1}]^3$, with $xyz \not=0$ and $|z|_p = \max\{ |x|_p, |y|_p, |z|_p \}$, such that
		$$|(z\alpha - x)(z\beta - y)|_p |z|_p^{2\epsilon + 7/4} < 1$$
		and $|x|_\infty < |x|_p^{1/4}$, $|y|_\infty < |y|_p^{1/4}$, $|z|_\infty < |z|_p^{1/4}$, $\min\{|x|_\infty,|y|_\infty\} > C$, for some constant $C > 0$, and $\min\{|x|_p,|y|_p\} > 1$.
	\end{lemma}
	\begin{proof}
		We introduce the following linearly independent forms in three variables
		\begin{equation*}
			L_{1,p}(x,y,z)=\alpha z - x \quad L_{2,p}(x,y,z)=\beta z - y, \quad
			L_{3,p}(x,y,z)= z
		\end{equation*}
		and
		\begin{equation*}
			L_{1,\infty}(x,y,z) = x, \quad  L_{2,\infty}(x,y,z) = y, \quad L_{3, \infty}(x,y,z) = z.
		\end{equation*}
		By Theorem \ref{roth}, the solutions $(x, y, z) \in \mathbb Z[p^{-1}]^3$ of
		\begin{equation} \label{eq:1}
			|(z\alpha - x)(z\beta - y)z|_p|xyz|_\infty\max(|x|_\infty, |y|_\infty, |z|_\infty)^\epsilon\max(|x|_p, |y|_p, |z|_p)^\epsilon < 1
		\end{equation}
		lie in a finite number of rational subspaces. Assume we have a solution $(x, y, z)$ in a fixed subspace  $x = Ay + Bz$ (with $A, B \in \mathbb Q$). Then formula  \eqref{eq:1} becomes
		\begin{equation} \label{eq:2}
			|(z(\alpha-B) - Ay)(z\beta -y) z|_p|(Ay+Bz)yz|\max(|Ay+Bz|, |y|, |z|)^\epsilon\max(|Ay+Bz|_p, |y|_p, |z|_p)^\epsilon < 1,
		\end{equation} 
		where  we denote by $|\cdot|$ what before was denoted by $|\cdot|_\infty$.
		
		Now observe that the solutions $(y,z) \in \mathbb Z[p^{-1}]^2$ of
		\begin{equation} \label{eq:3}
			|(z(\alpha-B) - Ay)(z\beta - y)|_p|(Ay+Bz)y|_\infty\max(|y|_\infty, |z|_\infty)^\epsilon\max(|y|_p, |z|_p)^\epsilon < 1
		\end{equation}
		lie, by Theorem \ref{roth}, in a finite number of rational subspaces, so that we can consider $z = Dy$, for $D \in \mathbb Q$, $D(A+BD)\neq0$ (recall that we discard solutions such that $D(A+BD)y^2=xz$ vanishes). Thus \eqref{eq:3} becomes
		\[|y|_p^{\epsilon + 2} |y|_\infty^{\epsilon+2} K < 1,\]
		for some constant $K \in \mathbb Q$, $K\neq0$.
		This inequality has a finite number of solutions in $\mathbb Z[p^{-1}]$ (given that $|y|_\infty > C$ and $|y|_p > 1$, with $C > 0$). Hence also \eqref{eq:3} has a finite number of solutions and thus the same holds for \eqref{eq:2} as well. Indeed, every solution of $\eqref{eq:2}$ is a solution of $\eqref{eq:3}$ (caveat, the converse may fail), since $z\in\mathbb Z[p^{-1}]$ implies $|z|_\infty |z|_p\geq 1$. 
		
		We can apply the same reasoning to solutions of \eqref{eq:1} lying in a subspace of the form $y=Ax+Bz$. This covers all cases we are interested in, because any other solution of \eqref{eq:1} would lie in the hyperplane $z=0$. Thus we have proved that there are finitely many solutions of \eqref{eq:1} in $\mathbb Z[p^{-1}]^3$ satisfying $xyz\neq0$ and $\min\{|x|_\infty,|y|_\infty\} > C$,  $\min\{|x|_p,|y|_p\} > 1$. 
		
		Now we use the hypothesis $|z|_p = \max(|x|_p, |y|_p, |z|_p)$. The left side of \eqref{eq:1} becomes
		\[F(x,y,z) := |(z \alpha - x)(z \beta - y)|_p|z|_p^{\epsilon + 1} |xyz|_\infty \max(|x|_\infty, |y|_\infty, |z|_\infty)^{\epsilon}\]
		and we have
		\[|(z\alpha - x)(z \beta - y)|_p|z|_p^{2\epsilon + 7/4} > F(x,y,z)\]
		for every $(x,y,z) \in \mathbb Z[p^{-1}]^3$ which satisfies the additional conditions $|x|_\infty < |x|_p^{1/4}$, $|y|_\infty < |y|_p^{1/4}$, $|z|_\infty < |z|_p^{1/4}$. This completes the proof.
	\end{proof}
	
	\begin{cor} \label{cor:1}
		Let $\alpha, \beta \in \mathbb Q_p$ be algebraic and $1, \alpha, \beta$ linearly independent over $\mathbb Q$. Then for every $\epsilon >0$ there are only finitely many $(x,y,z) \in \mathbb Z[p^{-1}]^3$, with $xyz \not=0$, such that
		\[ \left| \alpha - \frac{x}{z}  \right|_p < |z|_p^{-\epsilon - 15/8}, \quad \left| \beta - \frac{y}{z}\right|_p < |z|_p^{-\epsilon - 15/8}\]
		and $|x|_\infty < |x|_p^{1/4}$, $|y|_\infty < |y|_p^{1/4}$, $|z|_\infty < |z|_p^{1/4}$, $\min\{|x|_\infty,|y|_\infty\} > C$ (for some constant $C > 0$), $\min\{|x|_p,|y|_p\} > 1$, with $|z|_p = \max\{ |x|_p, |y|_p, |z|_p \}$.
	\end{cor}
	\begin{proof}
		Every triple $(x,y,z) \in \mathbb Z[p^{-1}]^3$ satisfying the hypotheses of the Corollary satisfies also the hypotheses of Lemma \ref{lemma:appl-sub}. Therefore there are finitely many such triples.
	\end{proof}
	
	\begin{cor}\label{coschmidt}
		Let $\alpha\in\mathbb Q_p$ be non rational and non--quadratic. Assume there exists a real number $\delta > 15/8$ and infinitely many $(x,y,z)\in\mathbb Z[p^{-1}]^3$, with $xyz \not=0$ and $|z|_p = \max\{ |x|_p, |y|_p, |z|_p \}$, such that
		\[
		\max\{|\alpha-x/z|_p, |\alpha^2-y/z|_p \}<\lvert z \rvert_p^{-\delta},
		\]
		and $|x|_\infty < |x|_p^{1/4}$, $|y|_\infty < |y|_p^{1/4}$, $|z|_\infty < |z|_p^{1/4}$, $\min\{|x|_\infty,|y|_\infty\} > C$ (for some constant $C > 0$), $\min\{|x|_p,|y|_p\} > 1$. Then $\alpha$ is transcendental.
	\end{cor}
	
	\begin{proof}
		Let us suppose $\alpha$ algebraic, non--rational and non--quadratic. Then $1, \alpha, \alpha^2$ are linearly independent over $\mathbb Q$ and the thesis follows by Corollary \ref{cor:1}.
	\end{proof}

	\section{Main results}
	\label{sec:main}
	
	In this section, we prove three results on the transcendence of Browkin $p$--adic continued fractions.
	We state the following hypothesis which contains the conditions that must be satisfied by the partial quotients in order that Lemma \ref{lemma:h2} holds. 
	
	\begin{hyp} \label{hyp:1}
		Let $(b_1, \ldots, b_k)$ be a finite sequence such that $b_i = \cfrac{\hat b_i}{p^{a_i}}$, with $a_i, \hat b_i \in \mathbb Z$, $a_i > 0$, $|\hat b_i|_\infty < \sqrt{\cfrac{3}{14}} \cdot \cfrac{p^a}{F_{k+1}}$ (for the $b_i \not= p^{-1}$), where $a = \min \{ a_i: 1 \leq i \leq k, a_i\not=1 \}$ and $(F_i)_{i \geq 0}$ the Fibonacci sequence.
	\end{hyp}
	
	\begin{theorem} \label{qper}
		Let $\alpha = [0, b_1, b_2, b_3, \ldots]$ be a non--periodic Browkin $p$--adic continued fraction such that $(b_i)$ is bounded in $\mathbb Q_p$, with $D=\max_i\{|b_i|_p\}$ and let $(n_i)_{i \geq 0}$, $(k_i)_{i \geq 0}$, $(\lambda_i)_{i \geq 0}$ be sequences of positive integers such that 
		\begin{itemize}
			\item the $k_i$'s are bounded;
			\item $n_{i+1} \geq n_i + \lambda_i k_i$ and there exists $C > 2\frac{\log D}{\log p} - 1$ such that $\lambda_i > C n_i$ for all $i$ sufficiently large;
			\item $b_{n_i} = \ldots = b_{n_i + \lambda_i k_i -1} = p^{-1}$ for every $i$; 
			\item the finite sequence $(b_1, \ldots, b_{n_{i-1}})$ satisfies Hypothesis \ref{hyp:1} for every $i\geq 1$.
		\end{itemize}
		Then $\alpha$ is transcendental.
	\end{theorem}
	
	\begin{proof}
		Let us suppose that $\alpha$ is an algebraic number, we define infinitely many irrational numbers of the kind
		\[ \beta^{(i)} = [0, b_1, \ldots, b_{n_i-1}, \overline{p^{-1}}], \]
		such that $h(\beta^{(i)}) \leq \lvert B_{n_i-1} \rvert_p^2$. The existence of infinitely many irrational numbers of this kind is ensured by Lemma \ref{lemma:h2} (taking $k = n_{i-1}$).
		By construction, the $\beta^{(i)}$ have the first $n_i+k_i\lambda_i$ partial quotients equal to the ones of $\alpha$, thus, by Lemma \ref{1}, we have
		\begin{equation} \label{eq:first-dis}
			\lvert \alpha - \beta^{(i)} \rvert_p < \lvert B_{n_i+k_i\lambda_i - 1} \rvert_p^{-2}.
		\end{equation}
		Moreover, we have
		\begin{equation}  \label{eq:second-dis}
			\lvert \alpha - \beta^{(i)} \rvert_p > \lvert B_{n_i - 1} \rvert_p^{-2\omega}
		\end{equation}
		with $\omega > 2$. Indeed, if $\lvert \alpha - \beta^{(i)} \rvert_p \leq \lvert B_{n_i - 1} \rvert_p^{-2\omega}$, then by Remark \ref{rem:H}, we have 
		\[ \lvert \alpha - \beta^{(i)} \rvert_p \leq h(\beta^{(i)})^{-\omega} \leq  C H(\beta^{(i)})^{-\omega} \]
		with $C > 0$, in contradiction with Corollary \ref{coroth}. Thus, using \eqref{eq:first-dis} and \eqref{eq:second-dis}, we obtain
		\begin{equation}  \label{eq:w}
			\lvert B_{n_i+k_i\lambda_i - 1} \rvert_p < \lvert B_{n_i-1} \rvert_p^{\omega}. 
		\end{equation}
		Since $|b_i|_p$ is bounded by $D$, we have $p \leq |b_i|_p \leq D$ for $i \geq 1$, from which we get $p^i \leq \lvert B_i \rvert_p \leq D^i$.
		Hence, considering $k_i$ bounded and remembering point (v) of Proposition \ref{prop:varie}, in order that \eqref{eq:w} holds, we must have
		\( \frac{\lambda_i}{n_i} < w \frac{\log D}{\log p} - 1. \)
		On the other hand, by hypothesis we have $\frac{\lambda_i}{n_i} > C$ and taking $$\delta = C - \frac{2 \log D}{\log p} + 1$$ we should have
		\[ 2 + \frac{\log p}{\log D}\delta < \omega \]
		for any $\omega > 2$, where $\delta >0$, which is not possible.
	\end{proof}
	
	\begin{theorem} \label{ooto}
		Let $\alpha = [0, b_1, b_2, b_3, \ldots]$ be a non--periodic Browkin $p$--adic continued fraction such that $(b_i)$ is bounded in $\mathbb Q_p$, with $D=\max_i\{|b_i|_p\}$ and let $(n_i)_{i \geq 0}$, $(k_i)_{i \geq 0}$, $(\lambda_i)_{i \geq 0}$ be sequences of positive integers such that 
		\begin{itemize}
			\item  the $k_i$'s are bounded;
			\item $n_{i+1} \geq n_i + \lambda_i k_i$ and there exists $C > 4\frac{\log D}{\log p} - 1$ such that $\lambda_i > C n_i$ for all $i$ sufficiently large;
			\item $b_{h+k_i} = b_{h}$, for $n_i \leq h \leq n_i + (\lambda_i-1)k_i - 1$, for every $i$;
			\item the finite sequence $(b_1, \ldots b_{n_i})$ satisfies Hypothesis \ref{hyp:1} for every $i$.
		\end{itemize}
		Then $\alpha$ is transcendental or a quadratic irrational.
	\end{theorem}
	
	\begin{proof}
		Let us suppose $\alpha$ an algebraic number of degree $> 2$.
		Since $(k_i)_{i \geq 0}$ is a bounded sequence of positive integers, there exist infinitely many $j$ such that
		\[k_j = k, \quad b_{n_j} = \hat b_{1}, \quad \ldots, \quad b_{n_j + k - 1} = \hat b_{k}\]
		for fixed $k \in \mathbb Z$ and $\hat b_1, \ldots \hat b_k \in \mathbb Z[p^{-1}]$. We define infinitely many irrational numbers of the kind
		\[ \beta^{(j)} = [0, b_1, \ldots, b_{n_j-1}, \overline{\hat b_1, \ldots, \hat b_k}]. \]
		By construction, the $\beta^{(j)}$’s have the first $n_j + k_j\lambda_j$ partial quotients equal to the ones of $\alpha$,
		thus, by Lemma 1, it follows
		\begin{equation}\label{eq:ab1}
			| \alpha - \beta^{(j)} |_p < |B_{n_j+k_j\lambda_j-1}|_p^{-2}.
		\end{equation}
		Moreover, we have
		\begin{equation} \label{eq:ab2}
			| \alpha - \beta^{(j)} |_p > |B_{n_j+k-1}|_p^{-4\omega}
		\end{equation}
		with $\omega > 2$. Indeed, if $| \alpha - \beta^{(j)} |_p \leq |B_{n_j+k-1}|_p^{-4\omega}$, then by Lemma \ref{lemma:h1} and Remark \ref{rem:H}, we have
		\[  | \alpha - \beta^{(j)} |_p \leq (2|B_{n_j+k-1}|_p^4)^{\varepsilon - \omega} \leq C H(\beta^{(i)})^{\varepsilon - \omega}, \]
		with $0 < \varepsilon < \omega - 2$ and $C > 0$, in contradiction with Corollary \ref{coroth}. Thus, using \eqref{eq:ab1} and \eqref{eq:ab2}, we obtain
		\begin{equation} \label{eq:ab3}
			|B_{n_j+k_j\lambda_j -1}|_p < |B_{n_j-k-1}|_p^{2\omega}.
		\end{equation}
		Since $|b_i|_p$ is bounded by D, we have $p \leq |b_i|_p \leq D$, for $i \geq 1$, from which we get $p^i \leq |B_i|_p \leq D^i$. Hence, in order that \eqref{eq:ab3} holds, there exists $a > 0$ such that
		\[ \lambda_j < a + \left( \frac{1}{2}(C+1)\omega-1 \right) n_j. \]
		Considering that $Cn_j < \lambda_j$ and $C > 4 \frac{\log D}{\log p} - 1$, we get
		\[ \left( 1 - \frac{\omega}{2} \right)(C+1) + \delta < \frac{a}{n_j} \]
		where $\delta = 4 \frac{\log D}{\log p} - 1$, which can not be satisfied for each $\omega > 2$ and $j$ sufficiently large.
	\end{proof}
	
	\begin{theorem}\label{pal}
		Assume $\alpha\in\mathbf Q_p$ has the Browkin continued fraction $\alpha = [0, b_1, b_2, \ldots]$, where $(b_n)_{n\ge 1}$ is a sequence beginning with arbitrarily long palindromes and $|b_n|_p \geq p^4$ for all $n \gg 0$. Also, assume that there exists a constant $C>0$ such that either $b_n>C\;\forall\,n$ or $-b_n<C\;\forall\,n$ holds. Then $\alpha$ is either transcendental or quadratic irrational.
	\end{theorem}
	
	
	\begin{proof}
		Let $n$ be a fixed natural number. Let $A_n/B_n$ be the $n$th convergent of $\alpha$, with $(A_i)$ and $(B_i)$ as in \eqref{eq:dfnAB}. By classical results on continued fractions, we have the following unique decomposition
		
		\[
		M_n=\begin{pmatrix}
			B_n & B_{n-1}\\
			A_n & A_{n-1}
		\end{pmatrix}
		= \begin{pmatrix}
			b_1 & 1 \\
			1   & 0
		\end{pmatrix}\begin{pmatrix}
			b_2 & 1 \\
			1 & 0
		\end{pmatrix}\cdots \begin{pmatrix}
			b_n & 1 \\
			1 & 0
		\end{pmatrix}.
		\]
		Indeed, by definition, since $A_n=b_nA_{n-1}+A_{n-2}$ and $B_n=b_nB_{n-1}+B_{n-2}$, we have
		
		\[
		M_n=M_{n-1}
		\begin{pmatrix}
			b_n & 1\\
			1 & 0    \end{pmatrix}
		\]
		and by a straightforward induction we obtain the required decomposition.
		Thus the matrix $M_n$ is symmetrical if and only if $(b_1,...,b_n)$ is palindromic. Assuming this is the case, we immediately see that  $A_n=B_{n-1}$. Then,  by Lemma \ref{1}, we have 
		$$\left\lvert \alpha-\cfrac{A_n}{B_n}\right\rvert_p < \cfrac{1}{|B_n|_p^2}  < \cfrac{|b_1|_p}{|B_n|_p^2}.$$
		Moreover, recalling that $A_n=B_{n-1}$ and $B_nA_{n-1}-A_nB_{n-1}=(-1)^n$, using Lemma \ref{1}, we obtain
		\begin{align*}
			\left\lvert \alpha^2 - \cfrac{A_{n-1}}{B_n}\right\lvert_p & = \left\lvert \alpha^2 -\cfrac{A_{n-1}}{B_{n-1}}\cfrac{A_{n}}{B_n}\right\lvert_p=\left\lvert\left(\alpha+ \cfrac{A_{n-1}}{B_{n-1}}\right)\left(\alpha-\cfrac{A_n}{B_n}\right)
			-\cfrac{(-1)^n\alpha}{B_nB_{n-1}}  \right\rvert_p
			\\ &\le \max\left\{ \left\lvert \alpha+ \cfrac{A_{n-1}}{B_{n-1}}\right\lvert_p \left\lvert \alpha-\cfrac{A_n}{B_n} \right\lvert_p ; \cfrac{|\alpha|_p}{|B_n B_{n-1}|_p} \right\}\\
			&\le \max\left\{\left\lvert \cfrac{\alpha}{B_n^2}\right\rvert_p ; \cfrac{|\alpha|_p}{|B_nB_{n-1}|_p} \right\} < \cfrac{|b_1|_p}{|B_n|_p^2}. 
		\end{align*}  where the first passage comes from $A_n=B_{n-1}$ and the inequalities in the last line hold because $|\alpha|_p < 1$ (since $b_0=0$) and 
		$$\cfrac{1}{|B_nB_{n-1}|_p} = \cfrac{|b_n|_p}{|B_n|_p^2} = \cfrac{|b_1|_p}{|B_n|_p^2}\,,$$
		remembering that we are  choosing $n$ so that $(b_1, \ldots, b_n)$ is a palindrome.
		Considering that $|b_n|_p \geq p^4$ for $n \gg 0$, by the proof of Lemma \ref{lemma:inf-p} we know that $|A_n|_\infty < |A_n|_p^{1/4}$, $|A_{n-1}|_\infty < |A_{n-1}|_p^{1/4}$ and $|B_n|_\infty < |B_n|_p^{1/4}$, for $n \gg 0$. Also, Proposition \ref{prop:varie} shows $|A_n|_p>1$ for $n$ large enough,  and $|B_n|_p = \max(|A_n|_p, |A_{n-1}|_p, |B_n|_p)$ for all $n$. Finally, a simple induction, using \eqref{eq:dfnAB}, shows
		\begin{equation} \label{eq:abn>c} \begin{cases} A_n>0\;\forall\,n\;\; & \text{ if }b_n>0\;\forall\,n;\\   A_{2n}<0\text{ and }A_{2n+1}>0\;\; & \text{ if }b_n<0\;\forall\,n. \end{cases} \end{equation}
		Since $b_0=0$, formula \eqref{eq:dfnAB} yields $A_1=1$, $A_2=b_2$,
		$$|A_3|_\infty=b_3b_2+1>C^2+1>C$$
		(thanks to our hypotheses on the $b_n$'s) and, assuming inductively that $|A_{n-1}|_\infty$, $|A_{n-2}|_\infty$ are both bigger than $C$,
		$$|A_n|_\infty=|b_nA_{n-1}|_\infty+|A_{n-2}|_\infty>C^2+C>C$$
		because of \eqref{eq:abn>c} and $|b_n|_\infty>C$. Thus, by taking $(x,y,z)=(A_n,A_{n-1},B_n)$ (with $n$ such that $(b_1, \ldots, b_n)$ is a palindrome) and $2>\delta>15/8$, we can apply Corollary \ref{coschmidt} to conclude. 
	\end{proof}
	
	
	\begin{rem}
		{\em
			The continued fraction expansion of a quadratic irrational can easily contain arbitrarily long palindromes, for example if it has a period of length two. On the other hand, recall that it is not known whether quadratic irrationals always have a periodic expansion (contrarily to the classical case). Hence, non-periodicity of $(b_i)_{i\geq1}$ is not enough to ensure transcendence.}
	\end{rem}
	
	\begin{exam}
		{\em
			Similarly to the archimedean case (see \cite{sturm}), we introduce $p$-adic $\mathit{Sturmian}$ continued fractions of slope $\theta$ as an instance of $p$-adic continued fractions beginning with arbitrarily large palindromes.
			Consider a real $\theta\in(0,1)$ and two distinct $a,b\in\mathbb Z[1/p]$.
			Then we define
			
			\[
			\sigma_\theta=[0,c_1,...]
			\]
			where 
			
			$$c_n= \begin{cases}
				a,& \text{ if } \lfloor (n+1)\theta\rfloor-\lfloor n\theta\rfloor=0
				
				\\ b,& \text{ if } \lfloor (n+1)\theta\rfloor-\lfloor n\theta\rfloor=1
			\end{cases}$$
			
			We also consider the $p$-adic $\mathit{Thue}$-$\mathit{Morse}$ continued fraction. A Thue-Morse sequence $(c_n)_n$ with values in $a,b$
			is defined by $c_n=a$ if the binary expansion of $n$ has an even  number of digits $1$, and $c_n=b$ otherwise. For instance, the word $c_0...c_{4^n-1}$ is clearly a palindrome. Consider the $p$-adic number $\theta=[0,c_0,...,c_n]$ where $c_i\in\{a,b\}$. If the sequence of partial quotients $(c_n)_n$ is a
			Thue-Morse word, then we call $\theta$ a $p$-adic Thue-Morse continued fraction.
			By Theorem \ref{pal} we immediately obtain the transcendence of both the $p$-adic Sturmian  continued fractions and the $p$-adic Thue-Morse continued fractions beginning with arbitrarily long palindromes.
		}
	\end{exam}

	\bibliographystyle{amsplain}

\end{document}